\documentclass[12pt]{amsart}
\usepackage{amscd,amssymb,amsthm,verbatim}
\usepackage{enumerate}
\usepackage{bm}

\usepackage{mathrsfs, graphicx} 
\usepackage{stmaryrd}

\usepackage[%
bookmarks=true,
colorlinks,
linkcolor=blue,
urlcolor=blue,
citecolor=blue,
plainpages=false,
pdfpagelabels,
final,
breaklinks=true\between
]{hyperref}
\usepackage{hyperref}
\usepackage[numbers,sort&compress]{natbib}

\numberwithin{equation}{section}

\theoremstyle{plain}
\newtheorem{thm}{Theorem}[section]
\newtheorem{prop}[thm]{Proposition}
\newtheorem{lem}[thm]{Lemma}
\newtheorem{cor}[thm]{Corollary}

\theoremstyle{definition}
\newtheorem{defn}[thm]{Definition}
\theoremstyle{remark}
\newtheorem{rk}[thm]{Remark}
\def\o{\overline}

\newcommand{\tr}{\operatorname{tr}}

\newcommand{\R}{\Bbb R}

\newcommand{\csum}{\mathbin{\#}}

\newcommand{\Inv}{\text{Inv}}

\DeclareFontFamily{U}{MnSymbolC}{}
\DeclareSymbolFont{MnSyC}{U}{MnSymbolC}{m}{n}
\DeclareFontShape{U}{MnSymbolC}{m}{n}{
	<-6>  MnSymbolC5
	<6-7>  MnSymbolC6
	<7-8>  MnSymbolC7
	<8-9>  MnSymbolC8
	<9-10> MnSymbolC9
	<10-12> MnSymbolC10
	<12->   MnSymbolC12}{}
\DeclareMathSymbol{\intprod}{\mathbin}{MnSyC}{'270}

\begin{document}
\title[Moduli Space of Asymptotically Flat Manifolds]{On the Moduli Space of Asymptotically Flat Manifolds with Boundary and the Constraint Equations}
\author{Sven Hirsch and Martin Lesourd}

\address{Black Hole Initiative, Harvard University, Cambridge, MA 02138}
\email{mlesourd@fas.harvard.edu}

\address{Department of Mathematics, Duke University, Durham, NC 27708-0320}
\email{sven.hirsch@duke.edu }

\maketitle

\begin{abstract}
Carlotto-Li have generalized Marques' path connectedness result for positive scalar curvature $R>0$ metrics on closed $3$-manifolds to the case of compact $3$-manifolds with $R>0$ and mean convex boundary $H>0$. Using their result, we show that the space of asymptotically flat metrics with nonnegative scalar curvature and mean convex boundary on $\mathbb{R}^3\backslash B^3$ is path connected. The argument bypasses Cerf's theorem, which was used in Marques' proof but which becomes inapplicable in the presence of a boundary. We also show path connectedness for a class of maximal initial data sets with marginally outer trapped boundary. 
\end{abstract}

\section{Introduction}
Let $\mathcal{M}_{R>0}$ be the space of positive scalar curvature metrics on $X$ endowed with the smooth topology, and $\text{Diff}(X)$ the group of diffeomorphisms acting on $X$. Marques \cite{marques2012deforming} proves the following fundamental result. 
\begin{thm}[Marques \cite{marques2012deforming}]\label{thm1}
Let $X$ be a closed 3-manifold admitting a metric of positive scalar curvature. Then $\mathcal{M}_{R>0}/ \emph{Diff}(X)$ is path connected. 
\end{thm}
Marques' beautiful proof combines Perelman's Ricci flow with surgery \cite{perelman2002entropy}, the conformal method, and Gromov-Lawson's gluing \cite{gromov1980classification}.\\ \indent 
Based on this and using a deep theorem of Cerf \cite{cerf1968diffeomorphismes} that implies the path connectedness of $\text{Diff}_+(D^3)$, the group of orientation preserving diffeomorphisms of the (closed) $3$-disc $D^3$, Marques proves that the space of asymptotically flat metrics on $\mathbb{R}^3$ with zero scalar curvature is path connected, and furthermore that the space of asymptotically flat vacuum and maximal solutions to the constraint equations of general relativity on $\mathbb{R}^3$ is path connected in a suitable topology. This improves the result of Smith-Weinstein \cite{smith2004quasiconvex}, who prove a similar result within a more restrictive class of metrics. \\ \indent
Recently, Carlotto-Li \cite{carlotto2019deformations} studied the space of positive scalar curvature metrics with mean convex boundary, $\mathcal{M}_{R>0,H>0}$, on compact 3-manifolds with boundary.\footnote{$H$ is measured pointing out of $X$.} After characterizing the topology of such manifolds, they prove the following fundamental generalization of Theorem \ref{thm1}.   
\begin{thm}[Carlotto-Li \cite{carlotto2019deformations}]\label{thm2}
Let $X$ be a compact 3-manifold with boundary that admits a metric with positive scalar curvature and mean convex boundary. Then $\mathcal{M}_{R>0,H>0}/ \emph{Diff}(X)$ is path connected.
\end{thm}
Their delicate argument proceeds in two stages. They first combine Miao's desingularization with the doubling of Gromov-Lawson to double the manifold across its boundary in a controlled way. They then study the resulting double using a sequence of equivariant Ricci flows. Their result extends to $\mathcal{M}_{R>0,H\geq0}/ \text{Diff}(X)$ and $\mathcal{M}_{R\geq0,H\geq0}/ \text{Diff}(X)$, though the latter behaves differently when $X$ is diffeomorphic to $S^1\times S^1\times I$. \\ \indent
In analogy with Marques' second result concerning $\mathbb{R}^3$, we show that Theorem \ref{thm2} implies the following, where $B^3$ denotes the open topological 3-ball. 
\begin{thm}\label{thm3}
Let $\mathcal{M}_{R=0,H=0}^{k,p,\rho}$ be the space of asymptotically flat metrics on $\mathbb{R}^3\backslash B^3$ as defined in Definition \ref{moduli1}. Then $\mathcal{M}_{R=0,H=0}^{k,p,\rho}$ is path connected in the $W^{k,p}_{\rho}$ topology.
\end{thm}
\begin{rk}
Thus given an asymptotically flat manifold with $R=0,H=0$, it is possible to continuously deform it to the Riemannian Schwarzschild manifold $(M_{S},g_{S})$, i.e., $M_{S}=\R^3\backslash B_1(0)$ with $g_{S}=(1+\frac 1r)^4g_E$, whilst maintaining $R=0,H=0$. Explicit paths of this kind to $g_S$ have been exploited to obtain geometric inequalities as in Bray's proof of the Riemannian Penrose inequality \cite{bray2001proof}.
\end{rk}
To prove Theorem \ref{thm3}, we construct a path between two arbitrary metrics in $\mathcal{M}_{R=0,H=0}^{k,p,\rho}$, which we denote by $g_{-1}$ and $g_2$. We apply Lemma \ref{deformation} to $g_{-1}$ and $g_2$ separately, to obtain a pair of smooth metrics $g_0$ and $g_1$ satisfying a list of desirable properties. The rest of the argument lies in connecting $g_0$ and $g_1$. This involves using Lemma \ref{compactification} to endow $S^3\backslash B^3$ (the compact model of $\mathbb{R}^3\backslash B^3$) with two metrics $\overline{g}_0,\overline{g}_1$ in such a way that $\overline{g}_0,\overline{g}_1$ can be joined by a continuous path of Yamabe positive metrics by \cite{escobar1992conformal} and Theorem \ref{thm2}. To finish the proof we need to invert the path on $S^3\backslash B^3$ to a path in $\mathcal{M}_{R=0,H=0}^{k,p,\rho}$ on $\mathbb{R}^3\backslash B^3$. This is done in Lemma \ref{interpolation}, which explicitly constructs the relevant diffeomorphisms.\\ \indent 

Carlotto-Li point out that their result implies a statement like Theorem \ref{thm3} by an argument along the lines of \cite{marques2012deforming}. Although this is true, Marques' argument alone does not strictly speaking yield the desired result for $\mathbb{R}^3\backslash B^3$. Towards the end of his proof, Marques employs the fact that $\text{Diff}_+(D^3)$ is path connected, which follows from \cite{cerf1968diffeomorphismes}. Running the same argument for $\mathbb{R}^3\backslash B^3$ eventually leads to one consider a possible path between two elements of $\text{Diff}_\partial (A^3)$, the group of boundary fixing diffeomorphisms of the closed $3$-annulus $A^3\simeq D^3_{r>r'}\backslash B^3_{r'}$. But as we show in Section \ref{section 2}, $\text{Diff}_\partial(A^3)$ is not path-connected and so the argument does not close. We get around this issue by finding a more constructive proof that does not invoke the homotopy type of diffeomorphism groups. 
\begin{rk}\label{remark1}
One wonders whether the arguments of Theorem \ref{thm3} apply to manifolds other than $\mathbb{R}^3\backslash B^3$, say $(\mathbb{R}^3\backslash B^3)\csum X'$ for $X'$ closed but not diffeomorphic to $S^3$, or say with multiple boundary components $\mathbb{R}^3\backslash \bigcup_{i=1} B^3_i$. Without modification however, the methods herein yield at best a result modulo quotients by the relevant diffeomorphism groups.
\end{rk}
\begin{rk}\label{remark3}
We note that if the initial asymptotically flat metric is smooth with $R=0,H=0$, then the path constructed in the proof of Theorem \ref{thm3} produces a continuous path of smooth metrics.
\end{rk}
Theorem \ref{thm3} gives the following.
\begin{cor}\label{cor}
Let $\mathcal{M}_{R\geq 0,H\geq0}^{k,p,\rho}$ be the space of asymptotically flat metrics on $\mathbb{R}^3\backslash B^3$ as defined in Definition \ref{moduli1}. Then $\mathcal{M}_{R\geq 0, H\geq 0}$ is path connected in the $W^{k,p}_\rho$ topology.
\end{cor}
As in \cite{marques2012deforming}, we also consider the constraint equations of general relativity, which are
\begin{equation}\label{eq:1}
16\pi\mu=R_{g}+(\text{Tr}_g\sigma)^2-|\sigma|_g^2,
\end{equation}
\begin{equation}\label{eq:2}
8\pi J=\text{div}_g\left(\sigma-(\text{Tr}_g\sigma) g\right)
\end{equation}
where $g$ is a metric and $\sigma$ a symmetric two-form on the manifold. We study the set of \textit{asymptotically flat} pairs $(g,\sigma)$ on $\mathbb{R}^3\backslash B^3$ solving \eqref{eq:1}, \eqref{eq:2} with $\mu=0, J=0$ (vaccum) and $\text{Tr}_g\sigma=0$ (maximal) along with the boundary condition
\begin{equation}\label{eq:3}
   \theta^+= -H+\text{Tr}_g(\sigma)-\sigma(\nu,\nu)=0,\quad  H\geq 0
\end{equation}
where $\nu$ is the normal to $\partial B^3$ pointing away from the asymptotically flat end. In general relativistic terms, we say that $\partial M$ is marginally outer trapped when $\theta^+=0$.
\begin{cor}\label{cor2}
Let $\mathcal{M}_{BH}$ be the space of pairs $(g,\sigma)$ on $\mathbb{R}^3\backslash B^3$ defined in Definition \ref{moduli1}, which satisfy \eqref{eq:1}, \eqref{eq:2}, \eqref{eq:3} with $\mu=0$, $J=0$, $\text{Tr}_g\sigma=0$. Then $\mathcal{M}_{BH}$ is path connected in the $W^{k,p}_\rho\oplus W^{k-1,p}_{\rho-1}$ topology.
\end{cor}
Corollary \ref{cor2} is relevant to general relativity because $\mathcal{M}_{BH}$ represents a certain class of initial data sets expected to give rise to a black hole spacetimes. The path connectedness of $\mathcal{M}_{BH}$ can be thought of as a necessary condition for the so-called Final State Conjecture, which states that generic black hole initial data sets asymptote to ones that isometrically embed into the Kerr solution. The conjecture may be studied for initial data sets belonging to $\mathcal{M}_{BH}$. Assuming that the subset $K\subsetneq \mathcal{M}_{BH}$ of pairs $(g,\sigma)$ that isometrically embed into the Kerr spacetime lie in at most one component of $\mathcal{M}_{BH}$, the disconnectedness of $\mathcal{M}_{BH}$ would imply that certain initial data sets describing black holes are unable to approach $K$ in a continuous way, which would spell trouble for the Final State Conjecture. Corollary \ref{cor2} shows that no such tension arises.\\

\textbf{Acknowledgements.}  S.H. and M.L. wish to thank Aghil Alaee, Hubert Bray, Jordan Keller, Sander Kupers, Alex Mramor, Robert Wald for useful discussions, and, in particular, Shing-Tung Yau for his interest and continuous support. M.L. would like to acknowledge the Gordon and Betty Moore and the John Templeton foundations for their support of the Black Hole Initiative. Finally, we would like to gratefully acknowledge the detailed report of the referee, which vastly improved the readability of the paper and led to a simplified argument.

\section{Definitions and Conventions}

$B^n$ is the topological open $n$-ball (not a metric ball), $D^n$ is the closed topological $n$-disc, and $A^n=D^n_{r>r'} \backslash B^n_{r'}$ is the closed $3$-annulus. \\ \\ \indent 
If $X$ is a manifold, $\text{Diff}_+(X)$ is the group of orientation preserving diffeomorphisms of $X$, and if $X$ has a boundary, $\text{Diff}_\partial (X)$ is the group of boundary fixing diffeomorphisms on $X$. \\ \\ \indent
Denote by $g_E$ the flat Euclidean metric on $\mathbb{R}^3$. For $x\in \mathbb{R}^3$, let $w(x)=(1+|x|^2)^{1/2}$. Then for any $\rho\in \mathbb{R}$, any $1<p <\infty$ and any open set $\Omega\subset \mathbb{R}^3$, the weighted Sobolev space  $W^{k,p}_{\rho}(\Omega)$ is the subset of $W^{k,p}_{\emph{loc}}(\Omega)$ for which the following norm is finite.
\begin{align}
||u||_{W^{k,p}_{\rho,g_E}(\Omega)} =\sum_{|\beta|\leq k}||w^{-\rho-\frac{n}{p}+|\beta|}\partial^\beta u ||_{L^p(\Omega)} 
\end{align}
Weighted spaces of continuous functions are defined by the norm 
\begin{align}
||u||_{C^k_{\rho,g_E}(\Omega)}= \sum_{|\alpha|\leq k} \sup_{x\in \Omega} w(x)^{-\rho+|\alpha|} |\partial^{\alpha}u(x) | 
\end{align}
and see \cite[Lemma 2.1]{maxwell2005solutions} for the standard embedding theorems in this context.\\ \indent 
Let $M=\R^3\backslash B^3$ be Euclidean space minus the unit ball.
Besides the Euclidean metric $g_E$, we suppose that $M$ is equipped with another Riemannian, metric $g\in W^{k,p}_{\emph{loc}}(M)$ such that $(M,g)$ is complete. 
Here $1/p -k/n <0$ so that $g$ is continuous.
We call $(M,g)$ asymptotically flat of order $\rho<0$ if $g-g_E\in W^{k,p}_\rho(M)$.\\ \indent
We denote by $W^{k,p}_{\rho}(M)=W^{k,p}_{\rho, g}(M)$ the corresponding weighted function spaces associated to the metric $g$. 
Analogously, we define the weighted function space $L^p_\rho(M)$ and $C^k_\rho(M)$, and $C^\infty_\rho(M)=\cap^\infty_{k=0} C^k_\rho(M)$. These weighted spaces are also serve to define an asymptotically flat initial data set for the constraint equations. The only additional requirement is that $\sigma$, the second fundamental form associated with the initial data set, behaves like a first derivative of $g\in W^{k,p}_{\rho}$ and thus belongs to $W^{k-1,p}_{\rho-1}(M)$. \\ \indent
A metric $g$ on an asymptotically flat manifold $M$ with boundary $\partial M$ is \textit{conformally flat} outside a compact set if there is a compact set containing $\partial M$ outside of which the metric takes the form $u^4 g_E$, where $g_E$ is the flat metric metric. \\ \indent 
Let us now define the main spaces of interest. 
\begin{defn}\label{moduli1}
Let $k$ be an integer $\geq 2$, let $-1<\rho<0$, and $p>3/k$. Then 
\begin{enumerate}
    \item[(i)] $\mathcal{M}_{R=0,H=0}^{k,p,\rho}$ denotes the set of metrics $g\in W^{k,p}_{\rho}$ on $M$ satisfying $R_g=0$, $H_{g}= 0$, 
    \item[(ii)] $\mathcal{M}_{R\geq0,H\geq0}^{k,p,\rho}$ is defined in the same way except that we enlarge the curvature restriction on $g$ from $H_g=0$ to $H_g\geq0$ and from $R_g=0$ to $R_g\geq 0$, 
    \item[(iii)] $\mathcal{M}_{BH}$ is defined as the set of pairs $(g\in W^{k,p}_\rho,\sigma\in W^{k-1,p}_{\rho-1})$ on $M$ that solve \eqref{eq:1}, \eqref{eq:2}, \eqref{eq:3} with $\mu=0,J=0$, $\text{Tr}_g\sigma=0$.
\end{enumerate}
\end{defn}

\section{Marques' Argument with a Boundary}\label{section 2}
After proving Theorem \ref{thm1} Marques turns to the case of asymptotically flat metrics on $\mathbb{R}^3$ and proves the path connectedness of three spaces, $\mathcal{M}_{R=0}$, $\mathcal{M}_{R\geq0}$, and $\mathcal{M}_{V}$ where $\mathcal{M}_{V}$ denotes the class of asymptotically flat pairs $(g,k)$ solving the vacuum, maximal constraint equations; Marques uses weighted H\"older spaces rather than Sobolev, but this distinction is immaterial here. The hardest part of the argument concerns $\mathcal{M}_{R=0}$ and proceeds in four stages. 
\begin{itemize}
\item[(1)] Based on work of Smith-Weinstein \cite{smith2004quasiconvex}, Marques constructs a continuous path from any metric in $\mathcal{M}_{R=0}$ to a smooth, harmonically flat metric $g\in \mathcal{M}_{R=0}$.
\item[(2)] With $g$ in hand, Marques shows the existence of a diffeomorphism $\phi:\mathbb{R}^3\to S^3\backslash \{p\}$ such that $\phi_*(g)=G^4\bar{g}$ where $\bar{g}$ is a metric on $S^3$ with positive Yamabe type and $G$ is the Green's function for the conformal Laplacian $\mathcal{L}_{\bar{g}}$, i.e., a function on $S^3\backslash \{p\}$ solving the distributional equation $\mathcal{L}_{\bar{g}}(G)=-4\pi \delta_p$.
\item[(3)] Using (2) and a theorem of Palais \cite{palais1959natural} on extensions of local diffeomorphisms, Marques constructs a continuous family of diffeomorphisms $\phi_\mu:\mathbb{R}^3\to S^3\backslash \{p\}$ with properties matching those in (2).
\item[(4)] The path connectedness of $\mathcal{M}_{R=0}$ results from combining (2) and (3) with Theorem \ref{thm1}. The key step is to show that the diffeomorphisms so far constructed can be composed to produce a diffeomorphism $F:\mathbb{R}^3\to \mathbb{R}^3$ which can be realized by a continuous path of diffeomorphisms $F_\mu:\mathbb{R}^3\to \mathbb{R}^3$.
\end{itemize}
To show (4), Marques relies on a theorem of Cerf \cite{cerf1968diffeomorphismes} which implies that $\text{Diff}_+(D^3)$ is path connected. It then remains to show that a diffeomorphism $F:\mathbb{R}^3\to \mathbb{R}^3$ which is identity outside a compact set can be realized by a continuous path $\mu\in [0,1]$ of diffeomorphisms $F_{\mu}:\mathbb{R}^3\to \mathbb{R}^3$, each identity outside a compact set, such that $F_{\mu=0}=\text{id}$ and $F_{\mu=1}=F$. Since each $F_{\mu}$ gives an orientation preserving diffeomorphism of $D^3$, the path connectedness of $\text{Diff}_+(D^3)$ guarantees the existence of the desired path $F_{\mu}$. \\ \\ \indent
So what of Marques' argument for $\mathbb{R}^3\backslash B^3$? \\ \\
At first sight, one can expect (1), (2), (3) to be suitably generalizable. But with regards to (4), one would seek to show that composing the diffeomorphisms constructed produces a diffeomorphism $F:\mathbb{R}^3\backslash B^3 \to \mathbb{R}^3\backslash B^3$ that is identity outside a compact set and identity within a neighborhood of $\partial B^3$, which permits preserving the mean curvature condition. To find a continuous path of diffeomorphisms playing the role of $F_{\mu}$, we thus consider $\text{Diff}_{\partial}(A^3)$. But at this point we run into the issue that $\text{Diff}_{\partial}(A^3)$ has two connected components - a fact that can be derived from Hatcher's Theorem \cite{hatcher1983proof}.\footnote{We thank Sander Kupers for showing us that Corollary \ref{fibration} can be shown without Theorem \ref{hatcher}. }
\begin{thm}[Hatcher's Theorem]\label{hatcher}
There is a weak homotopic equivalence $\emph{Diff}(S^3) \simeq O(4)$.
\end{thm}
\begin{cor}\label{fibration}
$\emph{Diff}_{\partial}(A^3)$ has two connected components.
\end{cor}
\begin{proof}
By Theorem \ref{hatcher} and the facts\footnote{See for instance \cite{hatcher1983proof} or more simply \cite[II]{diffeo}.} that there is a weak homotopic equivalence $\text{Diff}(S^3) \simeq O(4) \times \text{Diff}_{\partial}(D^3)$, and that $\text{Diff}_{\partial}(D^3) \to \text{Diff} (D^3) \to \text{Diff}(S^{3})$ is a fibration, it follows that $\text{Diff}_{\partial}(D^3)$ is weakly contractible. Now consider the action on the standard embedding $D^3_{r'} \hookrightarrow D^3_r$ by $\text{Diff}_{\partial}(D^3_r)$. Letting $\text{Emb}_+(D^3_{r'} \:, \: D^3_r)$ denote the space of orientation preserving embeddings of $D^3_{r'}$ into $D^3_{r>r'}$, this gives a fibration $\text{Diff}_{\partial}(D^3_r) \to \text{Emb}_+(D^3_{r'} \:, \: D^3_r)$ with fiber of the standard embedding given by $\text{Diff}_{\partial}(A^3)$. Putting these observations together, the long exact sequence of homotopy groups implies that $\pi_0 (\text{Diff}_{\partial}(A^3) ) = \pi_1 (\text{Emb}_+(D^3_{r'}\: , \: D^3_r) )$. Since the derivative at the origin gives a projection\footnote{See for instance the proof of \cite[Lemma 9.2.3]{diffeo}.} $\text{Emb}_+(D^3_{r'},\:D^3_r)\to SO(3)$, and since $\pi_1( SO(3) ) = \mathbb{Z}\backslash 2\mathbb{Z}$, it follows that $\text{Diff}_{\partial}(A^3)$ has two connected components. 
\end{proof} 
In sum, the best case seems to be that Marques' argument yields no more than two connected components.

\section{Proof}
There are three main steps to the proof of Theorem \ref{thm3}. From now on $M$ denotes a manifold diffeomorphic to $\mathbb{R}^3\backslash B^3$.
\begin{enumerate}
    \item \textbf{Deformation}. First we show Lemma \ref{deformation}, which gives a continuous path $\in \mathcal{M}_{R=0,H=0}^{k,p,\rho}$ from an arbitrary asymptotically flat metric to a metric which is conformally flat outside a compact set.
    \item \textbf{Compactification}. Then we show Lemma \ref{compactification}, which constructs a diffeomorphism from $M$, now endowed with a metric (conformally flat outside a compact set) with $R=0$ and $H=0$, to $(S^3\backslash B^3)\backslash \{p\}$ such that $S^3\backslash B^3$ admits a metric of positive Yamabe type.
    \item \textbf{Interpolation}. In Lemma \ref{interpolation} we explicitly construct a continuous family of diffeomorphisms from $M$ to $S^3\backslash B^3$ that permits combining steps (1) and (2) whilst invoking Theorem \ref{thm2} and \cite{escobar1992conformal}.
\end{enumerate} 
Steps (1), (2), (3) yield a proof of Theorem \ref{thm3}, whilst Corollaries \ref{cor}, \ref{cor2} follow relatively straightforwardly from Theorem \ref{thm3}.

\subsection{Deformation}

Before stating and proving Lemma \ref{deformation}, we recall the following result of Maxwell. Let $F_{\alpha,\beta}$ denote the following operator $((\Delta-\alpha)|_{M},(\partial_\nu+\beta)|_{\partial M}))$, and let $n$ be the dimension of $M$ ($3$ in our case). 
\begin{prop}[Maxwell \cite{maxwell2005solutions}]\label{maxwell}
Let $(M,g)$ be asymptotically flat of class $W^{k,p}_{\rho}$, $k\geq 2$, $k>n/p$, and suppose $\alpha\in W^{k-2,p}_{\rho-2}$ and $\beta\in W^{k-1-\frac{1}{p},p}$. Then if $2-n<\rho<0 $ the operator $F_{\alpha,\beta}:W^{k,p}_{\rho}\to W^{k-2,p}_{\rho-2}(M)\times W^{k-1-\frac{1}{p},p}(\partial M)$ is Fredholm with index $0$. Moreover if $\alpha,\beta\geq 0$ then $F$ is an isomorphism.
\end{prop}
We now prove the following.
\begin{lem}\label{deformation}
Let $g_{-1}\in \mathcal{M}_{R=0,H=0}^{k,p,\rho}$. There is a path $\mu\in [-1,0]\to g_\mu\in \mathcal{M}_{R=0,H=0}^{k,p,\rho}$ such that $g_0$ is smooth, conformally flat outside a compact set, minimal boundary, and moreover this path is continuous in the $W^{k,p}_\rho$ topology. 
\end{lem}
\begin{proof}
Let $\eta$ be a smooth cut-off function $0\leq \eta\leq 1$ such that $\eta(t)=1$ for $t\leq 1$ and $\eta(t)=0$ for $t\geq 2$.
Pick $R_0$ such that metric ball $B_{R_0}(0)$ contains $\partial M$. Given the metric $g_{-1}\in \mathcal{M}_{R=0,H=0}^{k,p,\rho}$, set $\eta_R(t)=\eta(t/R)$ for $R>R_0$, and define the new metric 
\begin{align}
g_{R}=(1-\eta_R)g_E+\eta_Rg_{-1}
\end{align}
We can now approximate $g_R$ with a smooth $g'_R$ such that $||g_R-g'_R||_{W^{k,p}_\rho}$ is arbitrarily small. \\ \\
For $\mu\in [-1,0]$, we define $g_{R,\mu}=(1+\mu)g'_{R}-\mu g_{-1}$. \\ \\ 
We now observe that by Proposition \ref{maxwell},  $F_{\frac{1}{8}R_{g_{-1}} ,\frac{1}{4}H_{g_{-1}}}$ is an isomorphism, and that consequently, for $||\hat{g}-g_{-1}||_{W^{k,p}_\rho}$ sufficiently small, $F_{\frac{1}{8}R_{\hat{g}},\frac{1}{4}H_{\hat{g}}}$ is also an isomorphism. So choosing $||g_R'-g_R||$ sufficiently small and $R$ large enough, $F_{\frac{1}{8}R_{g_{R,\mu}}, \frac{1}{4}H_{g_{R,\mu}}}$ is an isomorphism.  \\ \\
We may now consider the unique solution $v_{R,\mu}$ to $F_{\frac{1}{8}R_{g_{R,\mu}}, \frac{1}{4}H_{g_{R,\mu}}}=(\frac{1}{8}R_{g_{R,\mu}},-\frac{1}{4}H_{g_{R,\mu}})$. \\ \\
Observe that $\Delta v_{R,\mu}=\frac18R_{g_{R,\mu}}(1+v_{R_\mu})=\frac18R_{g_{R,\mu}}u$ where $u:=1+v_{R,\mu}$ and that $\partial_\nu v_{R,\mu}=-\frac{1}{4}H_{g_{R,\mu}}(1+v_{R,\mu})=-\frac{1}{4}H_{g_{R,\mu}}u$. The path in Lemma \ref{deformation} now follows by setting $g_\mu=u^4g_{R,\mu}$.
\end{proof}

\subsection{Compactification}
\begin{lem}\label{compactification}
Let $g$ be a smooth, asymptotically flat, conformally flat metric outside a compact set with $g=u^4g_E$, on $M\simeq \mathbb{R}^3\backslash B^3$ with $R= 0$ and $H=0$. Denote by $\overline{M}$ a manifold diffeomorphic $S^3\backslash B^3$ and $p$ a point in the interior of $\overline{M}$. Then there exists a diffeomorphism $\phi:M\to \overline{M}\backslash \{p\}$ and a smooth function $v:M \to \R$ such that 
\begin{itemize}
    \item outside a large ball in $M$ we have $\phi=\exp_{\overline{g},p}\circ \emph{Inv}$ where $\emph{Inv}(x):=\frac x{|x|^2}$ is the inversion map,
	\item on $\overline{M}\backslash \{p\}$ the metric $\overline{g}:=\phi_\ast(v^4g)$ extends smoothly to $\{p\}$,
	\item the mean curvature of $\partial \overline{M}$ satisfies $\overline{H}=0$,
	\item $(\overline{M},\overline{g})$ has positive Yamabe type.
\end{itemize}
\end{lem}
\begin{rk}
This is the boundary version of \cite[Theorem 9.3]{marques2012deforming}, which shows that an asymptotically flat manifold with positive scalar curvature gives rise to a Yamabe positive metric on $S^3$.
\end{rk}
\begin{proof}
Let $\phi:\R^3\to S^3\backslash \{p\}$ be the inverse stereographic projection which we restrict to $M$.
We define $v:M\to\R$ by
\begin{align}
\begin{cases}
v(x):=\frac 1{u(x)|x|} \quad\text{for $x$ near $\infty$},\\
v(x):=1\quad\text{for $x$ in a neighborhood of $\Sigma$}.
\end{cases}
\end{align}
where $\Sigma$ denotes the boundary of $\mathbb{R}^3\backslash B^3$. Clearly we can find $v(x)$ which smoothly interpolates between these two regions. \\ \indent 
By construction $\overline{g}$ can be extended smoothly from $\overline{M}\backslash \{p\}$ to $\overline{M}$ and we have $\overline{H}=H$ due to $v(x)=1$ in a neighborhood of $\Sigma$ .\\ \indent 
Moreover, it is easy to see that $\phi$ can be constructed such that $\phi=\exp\circ\text{Inv}$ outside a large ball.\\ \indent 
Thus it remains to show that $(\overline{M}, \overline{g})$ has positive Yamabe type.\\ \indent 
For this, we begin with showing that $G:=\phi_\ast(\frac1v)$ is a conformal Green's function on $(\overline{M}, \overline{g})$, that is 
\begin{align}
\label{eq:1'}
\mathcal L_{\overline{g}}G=-4\pi\delta_{p}\quad\quad\text{and}\quad\quad \partial_\nu G=-\frac {\bar H}4G=0.
\end{align}
Here $\mathcal L$ is again the conformal Laplacian $\mathcal L G=\Delta G - \frac18R_{\bar g}G$.
Observe that $\partial_\nu u =0$ and $\bar H=0$, so the second condition of being a conformal Green's function is satisfied.
Since the formula for scalar curvature under conformal transformation yields
\begin{align}
0=R=-G^5\mathcal L_{\overline{g}}G,
\end{align}
we obtain $\mathcal L_{\overline{g}}G=0$ outside $\{p\}$.
Near $\{p\}$ we have by construction $G(y)=\frac1{|y|}+\mathcal O(1)$ where $|y|$ denotes the distance from $p$ to $y$ with respect to $\overline{g}$.\\ \indent 
To show \eqref{eq:1'}, we first let $f$ be a smooth test function on $\overline{M}$. In that case, there exists, for every $\epsilon>0$, a $\delta\in (0,\epsilon]$ such that $|f(x)-f(p)|\le\epsilon$ for $x\in B_\delta(\infty)$.
Thus, we have
\begin{align}
\int_{B_\delta(p)}f\Delta G =&(f(p)+\mathcal O(\epsilon))\int_{S_\delta(p)}\partial_\nu G
\\=&(f(p)+\mathcal O(\epsilon))(-4\pi+\mathcal O(\delta^2))
\\=&-4\pi f(p)+\mathcal O(\epsilon)
\end{align}
which shows \eqref{eq:1'}.\\ \indent 
Escobar showed in \cite{escobar1992conformal} that for some $c\in\R$ there exists a solution to the Yamabe equation with boundary, i.e. there is a $\zeta:\overline{M}\to \R^+$ such that 
\begin{align}
\begin{cases}
\frac{R}{8}\zeta-\Delta\zeta=c\zeta^5\quad\text{for $x\in \overline{M}$},\\
\partial_\nu\zeta=-\frac H4\phi\quad \text{for $x\in\overline{\Sigma}$}.
\end{cases}
\end{align}
The Yamabe type is then characterized by the sign of $c$.
Thus, to show that $\overline{M}$ has positive Yamabe type it suffices to show that $c>0$.
We compute
\begin{align}
c\int_{\o M} G\zeta^5=&\int_{\o M}(-G\Delta \zeta+8GR\zeta)
\\=&-\int_{\o\Sigma} G\partial_\nu\zeta+\int_{\o M}\left(\left\langle \nabla\zeta,\nabla G\right\rangle+\frac18GR\zeta\right)\
\\=& \int_{\o \Sigma}\frac14\bar HG\zeta-\int_{\o M}\zeta\mathcal L_gG +\int_\Sigma \zeta\partial_\nu G\\
=&4\pi\zeta(p)>0.
\end{align}
This finishes the proof.
\end{proof}

\subsection{Interpolation}
It remains to show that the path $\mu\in[0,1]\to \overline{g}_\mu$ gives rise to a continuous path from $g_0$ to $g_1$ in $\mathcal{M}_{R=0,H=0}^{k,p,\rho}$. To do this, one needs to `invert' Lemma \ref{compactification} in a suitable way. This will be done using the usual blow up of $\overline{g}_\mu$ by its conformal Green's function $G_\mu$. This blow up gives a new metric on $\overline{M}\backslash \{p\}$, which in turn can be pulled back onto $M$ by a continuous family of diffeomorphisms $\phi_\mu:M\to \overline{M} \backslash \{p\}$. The challenge is now to construct $\{\phi_\mu\}$.

\begin{lem}\label{interpolation}
Let $\o g_0$ and $\o g_1$  be two metrics coming from Lemma \ref{compactification}, and let $\o g_\mu$ be a $C^\infty$-continuous path between these metrics. Moreover, let $G_\mu:\o M\to \R$ be a continuous family of functions which satisfy in a small neighborhood of $p$
\begin{align}
G_\mu(y)=\frac1{|y|_{\o g_\mu}}+A_{\mu}+\mathcal O_k(|y|_{\overline{g}_\mu})
\end{align}
where $|y|_{\o g_\mu}$ denotes the distance of $y$ to $p$ with respect to $\o g_\mu$. Then there exists a continuous path of diffeomorphisms $\phi_\mu$ such that $\mu\in [0,1]\to g_\mu\equiv \phi_\mu^\ast(G_\mu^4\o g_\mu)$ is a continuous path of asymptotically flat metrics $g_\mu\in C^\infty(M)\cap W^{k,p}_\rho(M)$ on $M$.
\end{lem}

\begin{proof}
Let $\phi:M\to\o M\backslash \{p\}$ be the diffeomorphism of Lemma \ref{compactification}.
We choose $\epsilon$ so small  such that
 \begin{itemize}
\item $\R^3\backslash B_{\frac1\epsilon}(0)$ is in the conformally flat regime of $g_0$ and $g_1$ so that $\o g_0=\o g_1=\psi_\ast g_E$ in a small neighborhood around $\{p\}$,
\item $\phi=\psi\circ \Inv$ outside $B_{\frac1\epsilon}(0)$ where $\psi=\exp_{\o g_1,p}=\exp_{\o g_0,p}=\exp_{g_E,p}$ in a small neighborhood around $\{p\}$,
\item $(\R^3\backslash B_{\frac1\epsilon}(0))\cap\Sigma=\emptyset$.
\end{itemize}
Our continuous family of diffeomorphisms $\phi_\mu$ can now be defined as follows 
\begin{align}
\phi_\mu=\begin{cases}
\phi\quad\text{inside $B_{\frac1\epsilon}(0)$}\\
\psi_\mu\circ\Inv\quad\text{outside $B_{\frac1\epsilon}(0)$}
\end{cases}
\end{align}
where $\psi_\mu:B_{\epsilon}(0)\hookrightarrow S^3$ is a continuous path of maps, each diffeomorphic onto their image, such that $\psi_0=\psi_1=\psi$ on $B_\epsilon(0)$. The task is now to show that $g_\mu=\phi_\mu^\ast(G^4_\mu\o g_\mu)$ is an asymptotically flat metric on $M$. To show this we must construct $\psi_\mu$ suitably. To do so we set $\psi_\mu \equiv \psi\circ f_\mu$ where $f_\mu:B_\epsilon(0)\to B_\epsilon(0)$ is a diffeomorphism onto its image with $f_0=f_1=\operatorname{Id}$. It now remains to specify $f_\mu$.\\ \indent
On $T_pS^3=\R^3$ we have both the standard metric $\psi_\ast g_E$ and the metric $\o g_\mu(p)$.
We choose an orthonormal basis $\{e_1,e_2,e_3\}_\mu$ for $g_E$ and rotate it such a way that $\o g_\mu(p)$ is diagonal, i.e. $\o g_\mu(p)=\operatorname{diag} (\lambda_1^2,\lambda_2^2,\lambda_3^2)_\mu$ for some numbers $\lambda_i>0$. Note that the choice of orthonormal basis is continuous with respect to $\mu$. Although $\{ e_1,e_2,e_3 \}_\mu$ and $(\lambda_1,\lambda_2,\lambda_3)_\mu$ depend on $\mu$, we hereby suppress this to declutter the notation. \\ \indent
We now define $f_\mu:B_\epsilon(0)\to B_\epsilon(0)$
\begin{align}
f_{\mu}(x_1,x_2,x_3)=\left(\frac {x_1}{\hat\lambda_1(r)},\frac {x_2}{\hat\lambda_2(r)}, \frac {x_3}{\hat\lambda_3(r)}\right)
\end{align}
and where $\hat\lambda_j$, $j=1,2,3$, are smooth, $\mu$-dependent and satisfy
\begin{align}
\hat\lambda_j(r)=\lambda_j
\end{align}
for $r=\sqrt{x_1^2+x_2^2+x_3^2}\le \xi$ where $0<\xi<\epsilon$ is chosen below. 
Moreover, we impose that
\begin{align}
\lambda_j(r)=1
\end{align}
for $\frac\epsilon2\le x\le \epsilon$ and that $\frac r{\hat\lambda_j(r)}$ is strictly monotone increasing. We can do this by choosing $\frac\xi{2\min\{\lambda_1,\lambda_2,\lambda_3\}}<\frac\epsilon2$. Since $[0,1]$ is compact, we can choose $\xi$ uniform in $\mu$. Moreover, $f_\mu:B_\epsilon(0)\to B_\epsilon(0)$ is a diffeomorphism that depends continuously on $\mu$, and $\psi_\mu$ is a diffeomorphism onto its image.\\ \indent 
It now remains to show that the resulting metric $g_\mu$ is asymptotically flat. Note that this follows not immediately since $\psi_\mu(x)\ne\exp_{\o g_\mu,p}(x)$.\\ \indent
% However, as it turns out error terms are  of sufficiently lower order so the error terms can be controlled.
For $y$ in a neighborhood of $p$, define a metric $\hat{g}$ by $\hat g(y)=\psi_\ast(\o g_\mu(p))$, where we have interpreted $\psi$ as map $\psi:B_\rho(0)\subset T_pS^3\to S^3$ for some sufficiently small $\rho$.\\ \indent
We now compare $\o g_\mu$ and $\hat g$. For this purpose we take our previous orthogonal basis with respect to $\o g_\mu(p)$, and orthonormal with respect to $g_E$ of $T_pS^3$ and scale it to be orthonormal with respect to $\o g_\mu(p)$; that is, we set $\o e_i=\frac1{\lambda_{i}}e_i$, $i=1,2,3$, and exponentiate it onto $S^3$ to obtain normal coordinates on $S^3$. 
By construction we have in these coordinate $\hat g_{ij}=\delta_{ij}$ and the normal coordinate formula yields  
\begin{align}\label{eq:normalcoordinates}
   (\o g_\mu)_{ij}(y)=\delta_{ij}+\mathcal O_2(|y|_{\o g_\mu}^2) 
\end{align}
in a small neighborhood around $\{p\}$.
Here the $\mathcal O_2$ denotes the higher order estimates
\begin{align}
\partial_k(\o g_\mu)_{ij}=\mathcal O(|y|),\quad \partial_k\partial_l(\o g_\mu)_{ij}=\mathcal O(1).
\end{align}
Furthermore, $\psi_\mu=\exp_{\hat g_\mu,p}$ in a small neighborhood of the origin. 
Next, we pull back the normal coordinates under $Df_\mu$ to obtain an orthonormal basis in $T_{y}\R^3$ for $|y|_{\o g_\mu}$ very small. After scaling and inverting, this gives then an orthonormal basis for $T_x\R^3$, $|x|$ very large, which we denote by $\{\tilde e_i\}$.
We compute
\begin{align}
g_{\mu}(\tilde e_i,\tilde e_j)=&((\psi\circ f_\mu\circ \Inv)^\ast(G^4_\mu\bar g_\mu ))(x)(\tilde e_i,\tilde e_j)\\
=&\frac1{|x|_{g_\mu}^4}\left(G_{\mu}\left(\psi\circ f_\mu\left(\frac x{|x|_{g_\mu}^2}\right)\right)\right)^4\o g_\mu\left(\frac x{|x|_{g_\mu}^2}\right)(\o e_i, \o e_j)
\end{align}
Our assumptions on $G$ together with equation \eqref{eq:normalcoordinates} yield in particular
\begin{align}
G_\mu(y)=\frac1{|y|_{\hat g_\mu}}+\mathcal O(1)
\end{align}
and 
\begin{align}
|\nabla G_\mu(y)|_{\o g_\mu}=\mathcal O(|y|_{\hat g_\mu}^{-2}),\quad|\nabla^2 G_\mu(y)|_{\o g_\mu}=\mathcal O(|y|_{\hat g_\mu}^{-3}).
\end{align}
Combining this with the equation
$
|\psi\circ f_\mu(\frac x{|x|_{g_\mu}^2})|_{\hat g_\mu}=\frac1{|x|_{g_\mu}}
$
implies
\begin{align}
\frac1{|x|^4_{g_\mu}}\left(G_\mu\left(\psi\circ f_\mu\left(\frac x{|x|_{g_\mu}^2}\right)\right)\right)^4=1+\mathcal O_2(|x|^{-1}_{g_\mu})
\end{align}
and thus
\begin{align}
g_\mu(\tilde e_i,\tilde e_j)=(1+\mathcal O_2(|x|^{-1}_{g_\mu}))(\delta_{ij}+\mathcal O_2(|x|^{-2}_{g_\mu}))=\delta_{ij}+\mathcal O_2(|x|^{-1}_{g_\mu}),
\end{align}
i.e. $g_\mu\in C^{\infty}(M)\cap C^2_{-1}(M)$ is asymptotically flat, which in turn implies $g_\mu\in W^{2,p}_\rho(M)$. Higher $k$ proceeds identically, and thus $g_\mu\in W^{k,p}_\rho(M)$.
\end{proof}

\subsection{Proof of Main Theorems}

\begin{proof}[Proof of Theorem \ref{thm3}]
Take any two metrics $g_{-1}$ and $g_2$ in $\mathcal{M}_{R=0,H=0}^{k,p,\rho}$ on $M$.
Our goal is to find a path $g_\mu$, $\mu\in[-1,2]$ in $\mathcal{M}_{R=0,H=0}^{k,p,\rho}$ connecting $g_{-1}$ and $g_2$.\\ \indent 
Lemma \ref{deformation} gives us a continuous path connecting $g_{-1}$ to $g_0$ and $g_2$ to $g_1$ respectively, with $g_0,g_1$ smooth, $R=0$, $H=0$, conformally flat outside a compact set containing $\partial M$. \\ \indent 
Lemma \ref{compactification} gives us a diffeomorphism from $(M,g_0)$ and $(M,g_1)$ onto $(S^3\backslash B^3)\backslash\{p\}$ that produces the manifold $(S^3\backslash B^3,\o g_0)$ and $(S^3\backslash B^3,\o g_1)$ where $\o g_0$ and $\o g_1$ are of positive Yamabe type with minimal boundary. \\ \indent 
Escobar's work on the Yamabe problem with boundary \cite{escobar1992conformal} guarantees that we can find a $C^\infty$-continuous path $\o g_\mu$, $\mu\in[0,1/3]$ within the conformal class of $\o g_0$ to a metric with positive scalar curvature metric and minimal boundary. 
Equally, we have a path from $\o g_1$ to $\o g_{2/3}$ with $\o g_{2/3}$ having positive scalar curvature and minimal boundary.\\ \indent
Theorem \ref{thm2} of Carlotto-Li gives us a path $\o g_\mu$, $\mu\in[1/3,2/3]$ between these metrics.
Next, we find as in \cite{escobar1992conformal} conformal Green's functions $G_\mu$, $\mu\in[0,1]$, i.e., functions satisfying $\partial_\nu G_\mu-\frac14H_{\o g_\mu}G_\mu=0$ and $\Delta_{\o g_\mu} G_\mu+\frac18 R_{\o g_\mu}G_\mu=4\pi\delta_p$ where $\delta_p$ is the Dirac $\delta$ function.
This places us in the setting of Lemma \ref{interpolation}, which allows us to lift the path $g_\mu$, $\mu\in[0,1]$ to a path $g_\mu$ of asymptotically flat metrics of zero scalar curvature and minimal boundary.
Hence, we have constructed a continuous path $g_\mu$ between $g_{-1}$ and $g_2$.\end{proof} 

\begin{figure}
    \centering
    \includegraphics[width=9.2cm]{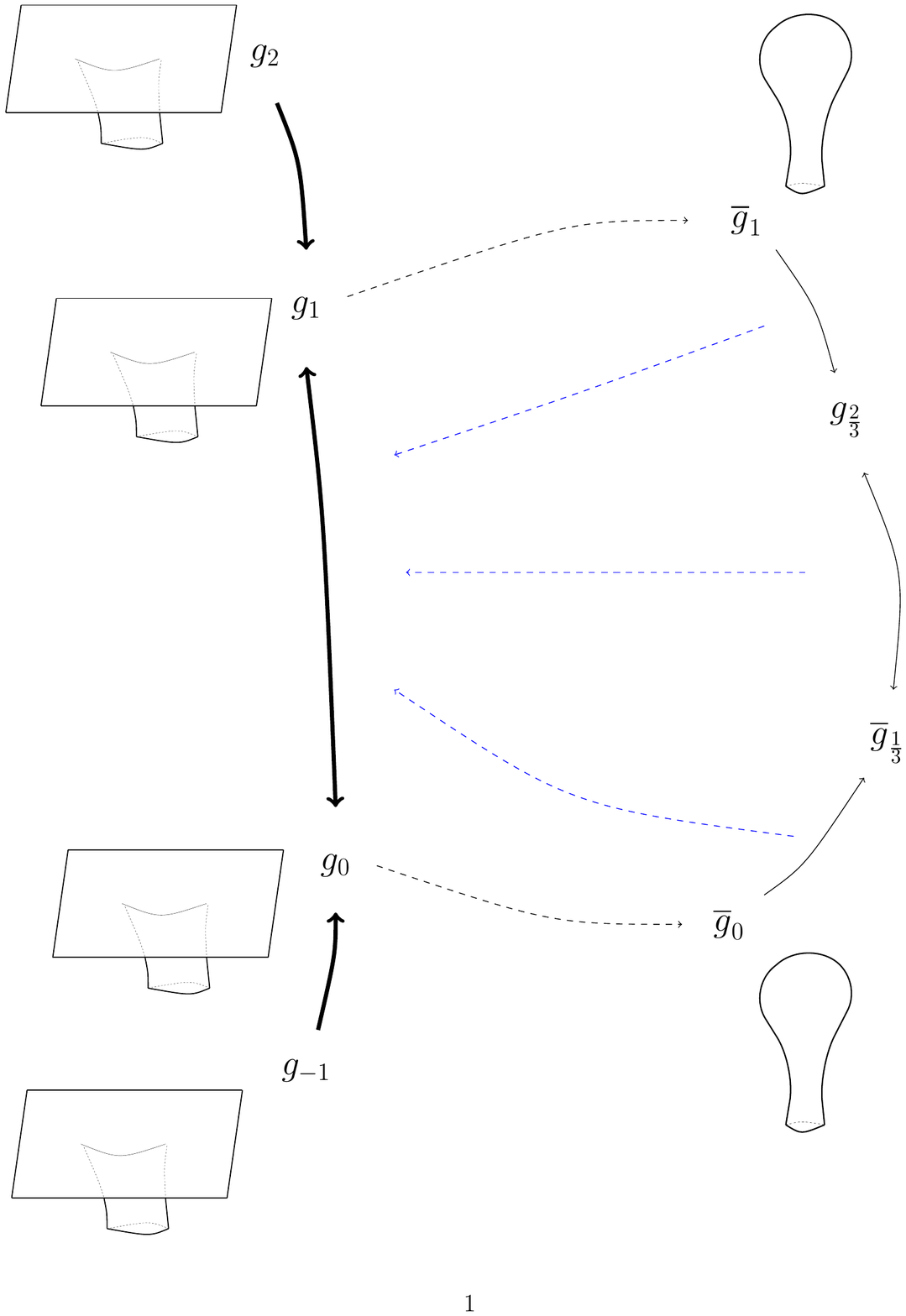}
    \caption{Lemma \ref{deformation} gives $g_2\to £g_1$ and $g_{-1}\to g_0$, and Lemma \ref{compactification} gives $g_1\to \overline{g}_1$ and $g_0\to \overline{g}_0$. The path from $\overline{g}_1$ to $\overline{g}_0$ follows from \cite{escobar1992conformal} and Theorem \ref{thm2}, and finally Lemma \ref{interpolation} gives the blue dashed line which permits lifting the path of metrics on $S^3\backslash B^3$ to a path of metrics on $M=\mathbb{R}^3\backslash B^3$.}
    \label{fig:my_label}
\end{figure}

\begin{proof}[Proof of Corollary \ref{cor}]
Consider the following PDE
\begin{align}
\begin{cases}
\Delta\phi-\frac18R\phi=0\quad\text{on $\R^3\setminus B^3$,}\\
\partial_\nu\phi+\frac14H\phi=0\quad\text{on $\partial B^3$}
\end{cases}
\end{align}
where $\phi\to1$ at $\infty$.
As in the proof of Lemma \ref{deformation}, we can use Proposition \ref{maxwell} to solve this equation.
Thus we can conformally deform both the scalar curvature and mean curvature to zero. This places us in the setting of Theorem \ref{thm3}, which finishes the proof.
\end{proof}

For Corollary \ref{cor2}, we start with the relevant space of interest, $\mathcal M_{BH}$, defined to be the set of all doubles $(g,\sigma)$ satisfying the vacuum, maximal constraint equations
\begin{align}
R=|\sigma|_g^2,\quad\quad\quad \tr_g\sigma=0, \quad \quad
\nabla_i\sigma^i_j=0
\end{align}
such that $g\in W^{k,p}_\rho$ is asymptotically flat, $\sigma \in W^{k-1,p}_{\rho-1}(M)$, and the surface $\Sigma\equiv \partial M$ satisfies $H+\sigma(\nu,\nu)\leq 0$ and $H\geq 0$ where $\nu$ is the normal to $\partial B^3$ pointing away from the asymptotically flat end..
\begin{proof}[Proof of Corollary \ref{cor2}]
Our task is to show that $\mathcal M_{BH}$ is path connected in $W^{k,p}_\rho(M) \oplus W^{k-1,p}_{\rho-1}(M) $. To do so we consider the larger set $\tilde{\mathcal M}_{BH}$ given by replacing $R=|\sigma|_g^2$ with $R\ge|\sigma|_g^2$ and $H\geq -\sigma(\nu,\nu)\geq 0$. Consider now the deformation 
\begin{align}
(g,\sigma)\to (g,(1-\mu)\sigma),\quad \mu\in[0,1]
\end{align}
Since $H\ge -\sigma(\nu,\nu)\ge0$  and $R\ge|\sigma|_g^2$ this deformation take place in $\tilde{\mathcal M}_{BH}$, from which it follows by Corollary \ref{cor} that $\tilde{\mathcal M}_{BH}$ is path connected.
We consider the conformal deformations $\hat g =u^4g$ and $\hat \sigma=u^{-2}\sigma$, which in turn implies $|\hat{\sigma}|^2_{\hat{g}}=u^{-12}|\sigma|^2_g$.
We would like to preserve the mean curvature condition $H\ge -\sigma(\nu,\nu)\ge0$ and the scalar curvature condition $R=|\sigma|_g^2$ under this conformal deformations.
%\textcolor{red}{Observe now that we can write}
%\begin{align}
%H=\xi(\sigma(\nu,\nu))
%\end{align}
%for some $\xi\ge1$. Consider the conformal deformations $\hat g =u^4g$ and $\hat \sigma=u^{-2}\sigma$. At the boundary we want to maintain $\hat H=\xi(\hat\sigma(\hat\nu,\hat\nu))=\xi(\sigma(\nu,\nu))u^2$. 
Recalling the formula for the change in mean curvature
\begin{align}
\hat H=u^{-2}H+4u^{-3}\partial_\nu u
\end{align}
and scalar curvature
\begin{align}
\hat R= u^{-4}R-8u^{-5}\Delta u
\end{align}
we are led to the Lichnerowicz equation
\begin{align}
\begin{cases}
\Delta u-\frac18Ru+\frac18|\sigma|^2u^{-7}=0 \quad\text{on $\R^3\setminus B^3$},\\
\partial_\nu u+\frac14H u-\frac 14 \sigma(\nu,\nu)u^{-3}=0\quad\text{on $\partial B^3$}
\end{cases}
\end{align}
where $u\to 1$ at $\infty$.\\ \indent
To solve this equation we use the method of sub and super solutions from \cite[Proposition 3.5]{maxwell2005solutions}. We note that $u=1$ is a supersolution. Next, as in the proof of Lemma \ref{deformation}, we solve the equation
\begin{align}
\begin{cases}
-\Delta \psi+\frac18R\psi=-\frac18R,\\
\partial_\nu \psi+\frac14H\psi=-\frac14H
\end{cases}
\end{align}
where $\psi\to0$ at $\infty$ and $\psi\in W^{k,p}_\rho$. 
Denoting $\underline u=1+\psi$ we have $0<\underline u\le 1$ and
\begin{align}
-\Delta \underline u+\frac18R\underline u-\frac18|\sigma|^2\underline u^{-7}=-\frac18|\sigma|^2\underline u^{-7}\le0.
\end{align}
Moreover, 
\begin{align}
&\partial_\nu \psi+\frac14H (1+\psi)-\frac 14H(1+\psi)^{-3}=-\frac14H(1+\psi)^{-3}\le0.
\end{align}
and so $\underline u$ is a subsolution, and thus there exists a solution $u\in 1+W^{k,p}_\rho$ with $0< u\leq 1$. \\ \indent 
Uniqueness now follows in a standard way. Namely, if $u_1$ and $u_2$ are solutions, then by the Lichnerowicz equation, $u_1-u_2$ satisfies \begin{align}
   \left( \left(\Delta -\frac{1}{8}\left(R-|\sigma|^2 \frac{u_1^{-7}-u_2^{-7}}{u_1-u_2}\right)\right), \left( \partial_\nu +\frac{1}{4}H\left(1-\frac{u_1^{-3}-u_2^{-3}}{u_1-u_2}\right) \right) \right) (u_1-u_2)=0
\end{align}
By the isomorphism of Proposition \ref{maxwell}, it follows that $u_1=u_2$. 
\end{proof}

\bibliographystyle{alpha}
\bibliography{bib.bib}

\end{document}